\documentclass[12pt]{amsart}
\usepackage{amsmath, amsthm, amscd, amsfonts, amssymb, color, graphicx, graphics}
\usepackage[mathlines]{lineno}
\usepackage[pagebackref=true, colorlinks, linkcolor=blue, citecolor=blue]{hyperref}

\newtheorem{thm}{Theorem}[section]
\newtheorem{lemma}[thm]{Lemma}

\newtheorem{cor}[thm]{Corollary}
\theoremstyle{definition}

\newtheorem{rem}[thm]{Remark}

\numberwithin{equation}{section}
\newfont{\kh}{msbm10}

\def\a{\alpha}
\def\b{\beta}
\def\g{\gamma}
\def\O{\Omega}
\def\o{\omega}
\def\t{\theta}
\def\z{\zeta}

\def\ve{\varepsilon}

\def\A{\mathcal{A}}
\def\I{\mathcal{I}}
\def\J{\mathcal{J}}
\def\E{\mathcal{E}}

\def\li{\mathrm{li}}
\def\d{\mathrm{d}}
\def\e{\mathrm{e}}

\def\Eh{\widehat{\mathcal{E}}}
\def\R{\widehat{R}}
\def\h{\widehat{h}}

\begin{document}

\title[Bounds for the number-theoretic omega functions]
{Global numerical bounds for the number-theoretic omega functions}

\author[M. Hassani]{Mehdi Hassani}

\address{
{\bf{Mehdi Hassani}}
\newline Department of Mathematics, University of Zanjan,  University Blvd., 45371-38791, Zanjan, Iran}

\email{mehdi.hassani@znu.ac.ir}

%%%%%%%%%%%%%%%%%%%%%



%\thanks{Thanks to University of Zanjan for supports.}

%\date{\today}

%\dedicatory{Dedicated to Professor Hari M. Srivastava}

%\commby{Prof GXGXXXX}

%%% ----------------------------------------------------------------------
\begin{abstract} 
We obtain global explicit numerical bounds, with best possible constants, for the differences $\frac{1}{n}\sum_{k\leqslant n}\o(k)-\log\log n$ and $\frac{1}{n}\sum_{k\leqslant n}\O(k)-\log\log n$, where $\o(k)$ and $\O(k)$ refer to the number of distinct prime divisors, and the total number of prime divisors of $k$, respectively.
\end{abstract}

\maketitle

\tableofcontents

\section{Introduction}
For the fixed complex number $s$ the generalized omega function $\O_s(k)$ is defined by $\O_s(k)=\sum_{p^\ell\|k}\ell^s$, where $p^\ell\|k$ means that $\ell$ is the largest power of $p$ such that $p^\ell|k$. The cases $s=0$ and $s=1$ coincide, respectively, with the well-known number theoretic omega functions $\o(k)=\sum_{p|k}1$, the number of distinct prime divisors of the positive integer $k$, and $\O(k)=\sum_{p^\ell\|k}\ell$, the total number of prime divisors of $k$. Duncan \cite{duncan} proved that for each arbitrary integer $s\geqslant 0$,
\begin{equation}\label{Duncan-approx}
\frac{1}{n}\sum_{k\leqslant n}\O_s(k)=\log\log n+M_s+O\left(\frac{1}{\log n}\right),
\end{equation}
where $M_s$ is a constant depending on $s$, given by $M_s=M+M'_s$, with $M$ referring to the Meissel--Mertens constant (see Remark \ref{M-remark} for more information), and 
\[
M'_s=\sum_{p}\sum_{\ell\geqslant 2}\frac{\ell^s-(\ell-1)^s}{p^\ell}.
\]
Here and through the paper $\sum_p$ means that the sum runs over all primes. Note that $M_0=M$. Also, we let $M'=M_1$ and $M'_1=M''=\sum_{p}\frac{1}{p(p-1)}$. Thus, $M'=M+M''$. Approximation \eqref{Duncan-approx} is a generalization of the previously known result of Hardy and Ramanujan \cite{hardy-rama} concerning the average of the functions $\o$ and $\O$. 

\medskip

Based on Dirichlet's hyperbola method and prime number theorem for arithmetic progressions with error term Saffari \cite{saffari} obtained a full asymptotic expansion for the average of $\o(n)$ where $n$ runs over the arithmetic progression $a$ modulo $q$ with $\gcd(a,q)=1$. For $a=q=1$, his result reads a follows
\begin{equation}\label{sum-omega-saffari-cor}
\frac{1}{n}\sum_{k\leqslant n}\o(k)=\log\log n+M+\sum_{j=1}^m\frac{a_j}{\log^j n}+O\Big(\frac{1}{\log^{m+1}n}\Big),
\end{equation}
where $m\geqslant 1$ is any fixed integer, and the coefficients $a_j$ are given by
\begin{equation}\label{aj}
a_j
=-\int_1^\infty\frac{\{t\}}{t^2}\log^{j-1} t\,\d t
=\frac{(-1)^{j-1}}{j}\,\frac{\d^j}{\d s^j}\left(\frac{1}{s}(s-1)\z(s)\right)_{s=1}.
\end{equation}
Diaconis \cite{diaconis} reproved \eqref{sum-omega-saffari-cor} using Dirichlet series of $\o$, Perron's formula and complex integration methods. One may obtain similar expansion for the average of generalized omega function $\O_s$ for each fixed real $s\geqslant 0$, replacing $M$ by $M_s$ (see \cite[Theorem 1]{hass-Omega-s} for more details).

\medskip

Explicit versions of \eqref{Duncan-approx} for $s=0$ and $s=1$ is obtained in \cite{hass-mia-omega} and \cite{hass-mia-Omega}, respectively, and then both improved in \cite[Theorem 1.2]{hass-Omega-omega-diff}, where it is showed that for each $n\geqslant 2$ the following double sided approximation holds
\begin{equation}\label{omega-bounds}
-\frac{1.133}{\log n}<\frac{1}{n}\sum_{k\leqslant n}\o(k)-\log\log n-M<\frac{1}{2\log^2 n}.
\end{equation}
Also,
\begin{equation}\label{Omega-bounds}
-\frac{1.175}{\log n}<\frac{1}{n}\sum_{k\leqslant n}\O(k)-\log\log n-M'<\frac{1}{2\log^2 n},
\end{equation}
where the left hand side is valid for each $n\geqslant 24$ and the right hand side is valid for each $n\geqslant 2$. 

\section{Summary of the results}

\subsection{Unconditional results} In the present paper we are motivated by finding global numerical lower and upper bounds for the differences $\A_0(n)$ and $\A_1(n)$, where $\A_s(n)$ defined for any fixed complex number $s$ as follows
\[
\A_s(n)=\frac{1}{n}\sum_{k\leqslant n}\O_s(k)-\log\log n.
\]
The problem for the case $\A_0(n)$ is an easy corollary of the inequalities \eqref{omega-bounds}. More precisely, we prove the following.

\begin{thm}\label{omega-bounds-thm} For all natural numbers $n\geqslant 2$, we have
\begin{equation}\label{omega-bounds-global}
\a_0\leqslant\A_0(n)\leqslant\b_0
\end{equation}
with the best possible constants $\a_0=\frac{45}{32}-\log\log 32$ and $\b_0=\frac{1}{2}-\log\log 2$, and the equality in the left hand side only for $n=32$, and in the right hand side only for $n=2$.
\end{thm}

Similarly, to get a global numerical lower bound for $\A_1(n)$, we can use the inequalities \eqref{Omega-bounds} to show the following result.

\begin{thm}\label{Omega-bounds-thm} For all natural numbers $n\geqslant 2$, we have
\begin{equation}
\a_1\leqslant\A_1(n)
\end{equation}
with the best possible constant $\a_1=\frac{8}{7}-\log\log 7$ and the equality only for $n=7$.
\end{thm}

The problem of obtaining a global numerical upper bound for $\A_1(n)$ is quite different by the above ones. Although, computations show that $\A_1(n)<\b_1$ for any $n\geqslant 2$ with the best possible constant $\b_1=M'$, but the inequalities \eqref{Omega-bounds} are not enough sharp to show this fact. To deal with this difficulty, we made explicit all steps of the proof of \eqref{sum-omega-saffari-cor} by following Saffari's argument in \cite{saffari}, and hence, we could to prove the following result.

\begin{thm}\label{A1-bound-thm}
For all natural numbers $n\geqslant\e^{14167}\approxeq 4.466\times 10^{6152}$, we have
\begin{equation}\label{A1-bound}
\A_1(n)<\b_1
\end{equation}
with the best possible constant $\b_1=M'$. Moreover, if we assume that the Riemann hypothesis is true, then \eqref{A1-bound} holds for all natural numbers $n\geqslant 1400387903260$.
\end{thm}

To prove Theorem \ref{A1-bound-thm} we use explicit forms of the prime number theorem with error term. Let $\pi(x)=\sum_{p\leqslant x}1$ be the prime counting function, and $\li(x)=\int_0^x\frac{1}{\log t}\,\d t$ be the logarithmic integral function, defined as the Cauchy principle value of the integral. By $f=O^\ast(g)$ we mean $|f|\leqslant g$, providing an explicit version of Landau's notation. It is known \cite[Theorem 2]{trudgian} that
\[
\pi(x)=\li(x)+O^\ast\left(0.2795\,x(\log x)^{-\frac{3}{4}}\,\e^{-\sqrt{(\log x)/6.455}}\right)\qquad(x\geqslant 229).
\]
Modifying the above to the classical form, for any $x>1.2$ we have
\begin{equation}\label{pi-li-bound}
\pi(x)=\li(x)+O^\ast(R(x)),\qquad R(x)=x\,\e^{-\frac{1}{3}\sqrt{\log x}}.
\end{equation}
This is however a weaker approximation, but it is suitable for our arguments, because of its global validity. We will use it to prove the following unconditional results.

\begin{thm}\label{sum-omega-full-explicit-thm}
For any fixed integer $m\geqslant 1$ and for any $x\geqslant\e$ we have
\begin{equation}\label{sum-omega-full-explicit}
\sum_{n\leqslant x}\o(n)=x\log\log x+Mx+x\sum_{j=1}^m\frac{a_j}{\log^j x}+O^\ast\left(\E_\o(x,m)\right),
\end{equation}
where
\begin{multline*}
\E_\o(x,m)=2^{m+1}m!\,\frac{x}{\log^{m+1} x}
+(2^{m+1}+1)\,\e m!\,\frac{\sqrt{x}}{\log x}\\
+x\,\e^{-\frac{\sqrt{2}}{6}\sqrt{\log x}}\left(\frac{1}{2}\log x+3\sqrt{2}\sqrt{\log x}+21\right)+\sqrt{x}.
\end{multline*}
\end{thm}

\begin{cor}\label{sum-omega-explicit-order-2-cor}
For $x\geqslant\e^{14167}\approxeq 4.466\times 10^{6152}$ we have
\begin{equation}\label{sum-omega-explicit-order-2}
\sum_{n\leqslant x}\o(n)=x\log\log x+Mx-\left(1-\g\right)\frac{x}{\log x}+O^\ast\left(\frac{5x}{\log^2x}\right),
\end{equation}
and consequently $\frac{1}{x}\sum_{n\leqslant x}\o(n)-\log\log x<M$.
\end{cor}

In order to transfer an average result on the function $\o$ to an average result on the function $\O$, we may consider the average difference $\J(x):=\sum_{n\leqslant x}\left(\O(n)-\o(n)\right)$, for which it is known \cite[Theorem 1.1]{hass-Omega-omega-diff} that for each integer $n\geqslant 1$,
\begin{equation}\label{J-explicit-bound-n}
nM''-25\frac{\sqrt{n}}{\log n}<\J(n)<nM''-\frac{\sqrt{n}}{\log n}\Big(2-\frac{20}{\log n}\Big).
\end{equation}
Modifying the above approximation, we will prove in Lemma \ref{J-explicit-bound-x-lemma} that $\J(x)=M''x+O^\ast(\frac{33\sqrt{x}}{\log x})$ for any $x\geqslant 2$. Thus, Theorem \ref{sum-omega-full-explicit-thm} and Corollary \ref{sum-omega-explicit-order-2-cor} transfer to the following results.

\begin{thm}\label{sum-Omega-full-explicit-thm}
For any fixed integer $m\geqslant 1$ and for any $x\geqslant\e$ we have
\begin{equation}\label{sum-Omega-full-explicit}
\sum_{n\leqslant x}\O(n)=x\log\log x+M'x+x\sum_{j=1}^m\frac{a_j}{\log^j x}+O^\ast\left(\E_\O(x,m)\right),
\end{equation}
where
\[
\E_\O(x,m)=\E_\o(x,m)+\frac{33\sqrt{x}}{\log x}.
\]
\end{thm}

\begin{cor}\label{sum-Omega-explicit-order-2-cor}
For $x\geqslant\e^{14167}\approxeq 4.466\times 10^{6152}$ we have
\begin{equation}\label{sum-Omega-explicit-order-2}
\sum_{n\leqslant x}\O(n)=x\log\log x+M'x-\left(1-\g\right)\frac{x}{\log x}+O^\ast\left(\frac{6x}{\log^2x}\right),
\end{equation}
and consequently $\frac{1}{x}\sum_{n\leqslant x}\O(n)-\log\log x<M'$.
\end{cor}

\subsection{Conditional results} As we observe in Corollary \ref{sum-omega-explicit-order-2-cor}, approximation \eqref{sum-omega-full-explicit}, even with its initial parameter $m=1$, gives explicit bounds for $\sum_{n\leqslant x}\o(n)$ for large values of $x$. The reason is using approximation \eqref{pi-li-bound} with the remainder term $R(x)$, and appearing the term $x\,\e^{-\frac{\sqrt{2}}{6}\sqrt{\log x}}$ in $\E_\o(x,m)$. This term comes essentially from the classical zero-free regions for the Riemann zeta function $\z(s)$. The situation changes as well, when we use approximations for $\pi(x)$ under assuming the Riemann hypothesis (RH), which asserts that $\Re(s)>\frac{1}{2}$ is a zero-free region, and indeed it is the best possible zero-free region, for $\z(s)$. Accordingly, it is known \cite[Corollary 1]{scho} that if the Riemann hypothesis is true, then 
\[
\pi(x)=\li(x)+O^\ast\left(\frac{1}{8\pi}\sqrt{x}\log x\right)\qquad(x\geqslant 2657).
\]
By computation, we observe that one may drop the coefficient $\frac{1}{8\pi}$ and get an easy to use bound for global range $x\geqslant 2$, as follows
\begin{equation}\label{pi-li-bound-RH}
\pi(x)=\li(x)+O^\ast\left(\R(x)\right),\qquad\R(x)=\sqrt{x}\log x.
\end{equation}
Note that the above approximations are close to optimal, because on one hand von Koch \cite{von-koch} showed that the Riemann hypothesis is equivalent to $\pi(x)=\li(x)+O(\sqrt{x}\log x)$, and on the other hand Littlewood \cite{litt} proved that letting $b(x)=\frac{\log\log\log x}{\log x}$, there are positive constants $c_1$ and $c_2$ such that there are arbitrarily large values of $x$ for which $\pi(x)>\li(x)+c_1\sqrt{x}\,b(x)$ and that there are also arbitrarily large values of $x$ for which $\pi(x)<\li(x)-c_2\sqrt{x}\,b(x)$. By using conditional approximation \eqref{pi-li-bound-RH}, we obtain the following analogues of Theorems \ref{sum-omega-full-explicit-thm}, \ref{sum-Omega-full-explicit-thm}, and Corollaries \ref{sum-omega-explicit-order-2-cor}, \ref{sum-Omega-explicit-order-2-cor}.

\begin{thm}\label{sum-omega-Omega-full-explicit-RH-thm}
Assume that the Riemann hypothesis is true. For any fixed integer $m\geqslant 1$ and for any $x\geqslant\e$ we have
\begin{equation}\label{sum-omega-full-explicit-RH}
\sum_{n\leqslant x}\o(n)=x\log\log x+Mx+x\sum_{j=1}^m\frac{a_j}{\log^j x}+O^\ast\left(\Eh_\o(x,m)\right),
\end{equation}
and
\begin{equation}\label{sum-Omega-full-explicit-RH}
\sum_{n\leqslant x}\O(n)=x\log\log x+M'x+x\sum_{j=1}^m\frac{a_j}{\log^j x}+O^\ast\left(\Eh_\O(x,m)\right),
\end{equation}
where
\begin{multline*}
\Eh_\o(x,m)=\left(\frac{3}{2}\right)^{m+1}m!\,\frac{x}{\log^{m+1} x}+4x^{\frac{2}{3}}\log x+9x^{\frac{2}{3}}\\+\left(\left(\frac{3}{2}\right)^{m+1}+1\right)\,\e m!\,\frac{x^{\frac{2}{3}}}{\log x}+15\sqrt{x}\log x,
\end{multline*}
and $\Eh_\O(x,m)=\Eh_\o(x,m)+\frac{33\sqrt{x}}{\log x}$.
\end{thm}

\begin{cor}\label{sum-omega-Omega-explicit-order-2-RH-cor}
Assume that the Riemann hypothesis is true, and let $x_0=1400387903260$. Then, for $x\geqslant x_0$ we have
\begin{equation}\label{sum-omega-explicit-order-2-RH}
\sum_{n\leqslant x}\o(n)=x\log\log x+Mx-\left(1-\g\right)\frac{x}{\log x}+O^\ast\left(\frac{11x}{\log^2x}\right),
\end{equation}
and
\begin{equation}\label{sum-Omega-explicit-order-2-RH}
\sum_{n\leqslant x}\O(n)=x\log\log x+M'x-\left(1-\g\right)\frac{x}{\log x}+O^\ast\left(\frac{12x}{\log^2x}\right),
\end{equation}
and consequently $\frac{1}{x}\sum_{n\leqslant x}\o(n)-\log\log x<M$ and $\frac{1}{x}\sum_{n\leqslant x}\O(n)-\log\log x<M'$.
\end{cor}

\begin{rem}
According to partial computations we could run, it seems that the inequality $\A_0(n)<M$ holds for $n\geqslant 16$, however, it fails for $n=15$. Also, as we mentioned above, the inequality $\A_1(n)<M'$ holds for any integer $n\geqslant 2$. A computational challenge is to check validity of them up to $x_0$, hence we will get a global conditional bound under RH. More generally, we ask about finding bounds for the difference $\A_s(n)$ for any fixed real $s>0$. A strategy to attack this problem is to make explicit the argument used in \cite{hass-Omega-s} to approximate the average difference $\J_s(n):=\sum_{k\leqslant n}\left(\O_s(k)-\o(k)\right)$, for which it is proved that 
\begin{equation*}\label{Js-bounds}
2^s\frac{\sqrt{n}}{\log n}\ll nM'_s-\J_s(n)\ll (2+\ve)^s\frac{\sqrt{n}}{\log n},
\end{equation*}
holds for each pair of fixed real numbers $s>0$ and $\ve>0$, and for $n$ sufficiently large.
\end{rem}

\begin{rem}\label{M-remark} The Meissel--Mertens constant $M$ \cite[pp. 94--98]{finch} is determined by 
\[
M=\g+\sum_{p}\left(\log\Big(1-p^{-1}\Big)+p^{-1}\right),
\]
where $\g$ is the Euler--Mascheroni constant \cite[pp. 24--40]{finch}. Also, see the impressive survey \cite{lagarias} for more information about $\g$. Among several properties of the constants $M$ and $M'$ we have the following rapidly converging series
\[
M=\g+\sum_{k=2}^\infty \frac{\mu(k)\log\z(k)}{k},
\quad\text{and}\quad
M'=\g+\sum_{k=2}^\infty \frac{\varphi(k)\log\z(k)}{k},
\]
where $\mu$ is the M\"{o}bus function and $\varphi$ is the Euler function. Computations based on the above series representations yields that
\begin{align*}
M&\approxeq 0.26149721284764278375542683860869585905156664826120,\\
M'&\approxeq 1.03465388189743791161979429846463825467030798434439.
\end{align*}
We have used these values in our numerical verifications of the results of the present paper. All of computations have been done over Maple software\footnote{We mention that the Maple command to compute $\O(n)$ is \textcolor{red}{\texttt{bigomega(n)}} and accordingly, a Maple code to compute $\o(n)$ is given by\\
\textcolor{red}{\texttt{with(numtheory):}}\\
\textcolor{red}{\texttt{rad:= n -> convert(numtheory:-factorset(n), `*`):}}\\
\textcolor{red}{\texttt{smallomega:=n->bigomega(rad(n));}}
}.
\end{rem}

\section{Proof of unconditional approximations}

\begin{proof}[Proof of Theorem \ref{omega-bounds-thm}] Considering the left hand side of \eqref{omega-bounds}, we observe that the inequalities 
\[
\A_0(n)>M-\frac{1.133}{\log n}>\a_0
\]
hold when $n>\e^{1.133/(M-\a_0)}\approxeq 102841.56$. Thus, we obtain the left hand side of \eqref{omega-bounds-global} for any integer $n\geqslant 102842$. By computation, it holds also for $2\leqslant n\leqslant 102841$ with equality only for $n=32$. Also, considering the right hand side of \eqref{omega-bounds}, we observe that the inequalities
\[
\A_0(n)<M+\frac{1}{2\log^2 n}<\b_0
\]
hold when $n>\e^{1/\sqrt{2(\b_0-M)}}\approxeq 2.48$. This completes the proof.
\end{proof}

\begin{proof}[Proof of Theorem \ref{Omega-bounds-thm}] Since $\e^{1.175/(M'-\a_1)}\approxeq 8.23$, for any integer $n\geqslant 9$ we have $n>\e^{1.175/(M'-\a_1)}$, or equivalently $M'-1.175/\log n>\a_1$. By using this inequality, and the left hand side of \eqref{Omega-bounds} we deduce that $\A_1(n)>\a_1$ holds for $n\geqslant 24$. By computation, it holds also for $2\leqslant n\leqslant 24$ with equality only for $n=7$. This completes the proof.
\end{proof}

Proof of Theorems \ref{sum-omega-full-explicit-thm} and \ref{sum-Omega-full-explicit-thm} and their corollaries based on a series of lemmas. As in \cite{saffari}, we start by using Dirichlet's hyperbola method \cite[Theorem 3.1]{tenenbaum} to get the following result.

\begin{lemma}
For any $x$ and $y$ satisfying $1\leqslant y\leqslant x$, we have
\begin{equation}\label{sum-omega-Dirichlet}
\sum_{n\leqslant x}\o(n)
=\sum_{p\leqslant y}\left[\frac{x}{p}\right]+\sum_{n\leqslant\frac{x}{y}}\pi\left(\frac{x}{n}\right)-\left[\frac{x}{y}\right]\pi(y).
\end{equation}
\end{lemma}
\begin{proof}
Let $\mathbf{1}(n)=1$ be the unitary arithmetic function, and $\varpi(n)$ be the characteristic function of primes; that is $\varpi(n)=1$ when $n$ is prime, and $\varpi(n)=0$ otherwise. We consider Dirichlet convolution of these two functions,
\[
\mathbf{1}\ast\varpi(n)=\varpi\ast\mathbf{1}(n)
=\sum_{d|n}\varpi(d)\,\mathbf{1}\left(\frac{n}{d}\right)
=\sum_{d|n}\varpi(d)=\sum_{p|n}1=\o(n).
\]
Note that $[x]=\sum_{n\leq x}\mathbf{1}(n)$, and $\pi(x)=\sum_{n\leqslant x}\varpi(n)$. Thus, by using Dirichlet's hyperbola method, for any $y$ satisfying $1\leqslant y\leqslant x$ we deduce that
\[
\sum_{n\leqslant x}\o(n)
=\sum_{n\leqslant x}\mathbf{1}\ast\varpi(n)
=\sum_{n\leqslant y}\left[\frac{x}{n}\right]\varpi(n)+\sum_{n\leqslant\frac{x}{y}}\pi\left(\frac{x}{n}\right)-\left[\frac{x}{y}\right]\pi(y).
\]
This gives \eqref{sum-omega-Dirichlet}.
\end{proof}

\begin{lemma}\label{sum-p-y-x-p-order-infinity-lemma}
For any $x$ and $y$ satisfying $1.2<y\leqslant x$, we have
\begin{equation}\label{sum-p-y-x-p-order-infinity}
\sum_{p\leqslant y}\left[\frac{x}{p}\right]=x\log\log y+Mx+O^\ast\left(h_1(x,y)\right),
\end{equation}
where
\begin{equation*}\label{h1xy}
h_1(x,y)=x\,\e^{-\frac{1}{3}\sqrt{\log y}}\left(6\sqrt{\log y}+19\right)+y.
\end{equation*}
\end{lemma}
\begin{proof}
We have
\[
\sum_{p\leqslant y}\left[\frac{x}{p}\right]
=\sum_{p\leqslant y}\left(\frac{x}{p}-\left\{\frac{x}{p}\right\}\right)
=x\sum_{p\leqslant y}\frac{1}{p}+O^\ast(y).
\]
The Stieltjes integral and integration by parts gives
\begin{align*}
\sum_{p\leqslant y}\frac{1}{p}
&=\int_{2^-}^y\frac{\d\pi(t)}{t}
=\frac{\pi(y)}{y}+\int_2^y\frac{\li(t)}{t^2}\,\d t+\int_2^y\frac{\pi(t)-\li(t)}{t^2}\,\d t\\
&=\frac{\li(y)}{y}+O^\ast\left(\frac{R(y)}{y}\right)+\int_2^y\frac{\li(t)}{t^2}\,\d t+\int_2^y\frac{\pi(t)-\li(t)}{t^2}\,\d t.
\end{align*}
The last integral is dominated by $\int_2^\infty\frac{R(t)}{t^2}\,\d t$, so it is convergent as $y\to\infty$. Thus, we have
\[
\int_2^y\frac{\pi(t)-\li(t)}{t^2}\,\d t=\int_2^\infty\frac{\pi(t)-\li(t)}{t^2}\,\d t+O^\ast\left(\int_y^\infty\frac{R(t)}{t^2}\,\d t\right).
\]
Note that
\[
\int_y^\infty\frac{R(t)}{t^2}\,\d t=\e^{-\frac{1}{3}\sqrt{\log y}}\left(6\sqrt{\log y}+18\right).
\]
Also, integration by parts implies
\[
\int_2^y\frac{\li(t)}{t^2}\,\d t=-\frac{\li(t)}{t}\Big|_2^y+\int_2^y\frac{\d t}{t\log t}=\log\log y-\frac{\li(y)}{y}+\frac{\li(2)}{2}-\log\log 2.
\]
Combining the above approximations, we deduce that
\[
\sum_{p\leqslant y}\frac{1}{p}=\log\log y+C+O^\ast\left(\e^{-\frac{1}{3}\sqrt{\log y}}\left(6\sqrt{\log y}+19\right)\right),
\]
where 
\[
C=\int_2^\infty\frac{\pi(t)-\li(t)}{t^2}\,\d t+\frac{\li(2)}{2}-\log\log 2.
\]
Mertens' approximation concerning the sum of reciprocal of primes \cite[Theorem 1.10]{tenenbaum} asserts that $\sum_{p\leqslant y}\frac{1}{p}-\log\log y\to M$ as $y\to\infty$. This implies that $C=M$, and concludes the proof. Meanwhile, let us mention that the equality $C=M$ also implies that
\[
\int_2^\infty\frac{\pi(t)-\li(t)}{t^2}\,\d t
=M+\log\log 2-\frac{\li(2)}{2}\approxeq -0.62759759779276794.
\]
An additional output of the completed proof.
\end{proof}

\begin{lemma}\label{sum-pi-x-n-lemma}
Let $x$ and $y$ satisfy $x\geqslant\e$ and $1.2<x^\delta\leqslant y\leqslant x^\Delta<x$ for some fixed $\delta, \Delta\in(0,1)$. Then, we have
\begin{multline}\label{sum-pi-x-n}
\sum_{n\leqslant\frac{x}{y}}\pi\left(\frac{x}{n}\right)
=\left[\frac{x}{y}\right]\li(y)+x(\log\log x-\log\log y)\\+x\sum_{j=1}^m\frac{a_j}{\log^j x}+O^\ast\left(h_2(x,y)\right),
\end{multline}
where
\begin{multline*}\label{h2xy}
h_2(x,y)=\frac{m!}{\delta^{m+1}}\,\frac{x}{\log^{m+1} x}\\+\left(1+\frac{1}{\delta^{m+1}}\right)\e m!\,\frac{x^\Delta}{\log x}+x\,\e^{-\frac{1}{3}\sqrt{\log y}}\left(1+\log\frac{x}{y}\right).
\end{multline*}
\end{lemma}
\begin{proof}
For $n\leqslant\frac{x}{y}$ we have $\frac{x}{n}\geqslant y\geqslant x^\delta>1.2$. Thus, we may use the approximation \eqref{pi-li-bound} to get
\[
\sum_{n\leqslant\frac{x}{y}}\pi\left(\frac{x}{n}\right)
=\sum_{n\leqslant\frac{x}{y}}\li\left(\frac{x}{n}\right)+O^\ast\left(\sum_{n\leqslant\frac{x}{y}}R\left(\frac{x}{n}\right)\right).
\]
Since $\frac{\d}{\d t}\li\left(\frac{x}{t}\right)=-\frac{x}{t^2(\log x-\log t)}$, the Stieltjes integral and integration by parts gives
\[
\sum_{n\leqslant\frac{x}{y}}\li\left(\frac{x}{n}\right)
=\int_{1^-}^\frac{x}{y}\li\left(\frac{x}{t}\right)\d [t]
=\left[\frac{x}{y}\right]\li(y)+x\int_1^\frac{x}{y}\frac{[t]}{t^2(\log x-\log t)}\,\d t.
\]
We write $[t]=t-\{t\}$ to get
\begin{equation}\label{sum-li-x-y-E}
\sum_{n\leqslant\frac{x}{y}}\li\left(\frac{x}{n}\right)
=\left[\frac{x}{y}\right]\li(y)+x(\log\log x-\log\log y)-\E(x,y),
\end{equation}
with the remainder $\E(x,y)$ given by
\[
\E(x,y)=x\int_1^\frac{x}{y}\frac{\{t\}}{t^2(\log x-\log t)}\,\d t.
\]
Letting $g_x(t)=(1-\frac{\log t}{\log x})^{-1}$, we have
\[
\E(x,y)
=\frac{x}{\log x}\int_1^\frac{x}{y}\frac{\{t\}}{t^2}\,g_x(t)\,\d t
=\E_1(x,y)-\E_2(x,y),
\]
with
\[
\E_1(x,y)=\frac{x}{\log x}\int_1^\infty\frac{\{t\}}{t^2}\,g_x(t)\,\d t,
\qquad
\E_2(x,y)=\frac{x}{\log x}\int_\frac{x}{y}^\infty\frac{\{t\}}{t^2}\,g_x(t)\,\d t.
\]
Since $y\geqslant x^\delta$, we have $1\leqslant t\leqslant\frac{x}{y}\leqslant x^{1-\delta}$, and consequently $0\leqslant\frac{\log t}{\log x}\leqslant 1-\delta<1$. We use Taylor's formula with remainder \cite[Theorem 5.19]{apostol} for the function $u\mapsto (1-u)^{-1}$, which asserts that if $0\leqslant u\leqslant 1-\delta$ for some fixed $\delta\in(0,1)$, as in our case, then for any given integer $m\geqslant 1$,
\begin{equation}\label{truncated-geometric-series}
(1-u)^{-1}=\sum_{r=0}^{m-1} u^r+O^\ast\left(\frac{1}{\delta^{m+1}}\,u^{m}\right).
\end{equation}
Taking $u=\frac{\log t}{\log x}$ in \eqref{truncated-geometric-series}, we get
\begin{equation}\label{gxt-truncated-geometric-series}
g_x(t)=\sum_{r=0}^{m-1}\left(\frac{\log t}{\log x}\right)^r+O^\ast\left(\frac{1}{\delta^{m+1}}\left(\frac{\log t}{\log x}\right)^{m}\right).
\end{equation}
Thus,
\[
\E_1(x,y)
=\frac{x}{\log x}\int_1^\infty\frac{\{t\}}{t^2}\sum_{r=0}^{m-1}\left(\frac{\log t}{\log x}\right)^r\,\d t+h_\delta(x),
\]
where
\begin{align*}
|h_\delta(x)|
&\leqslant\frac{x}{\log x}\int_1^\infty\frac{\{t\}}{t^2}\frac{1}{\delta^{m+1}}\left(\frac{\log t}{\log x}\right)^{m}\,\d t\\
&\leqslant\frac{1}{\delta^{m+1}}\,\frac{x}{\log^{m+1} x}\int_1^\infty\frac{\log^m t}{t^2}\,\d t
=\frac{m!}{\delta^{m+1}}\,\frac{x}{\log^{m+1} x}.
\end{align*}
Also, we have
\begin{multline*}
\frac{x}{\log x}\int_1^\infty\frac{\{t\}}{t^2}\sum_{r=0}^{m-1}\left(\frac{\log t}{\log x}\right)^r\,\d t
\\=\sum_{j=1}^m\frac{x}{\log^j x}\int_1^\infty\frac{\{t\}}{t^2}\log^{j-1} t\,\d t
=-x\sum_{j=1}^m\frac{a_j}{\log^j x}.
\end{multline*}
Hence, the following approximation holds for any fixed integer $m\geqslant 1$, with the coefficients $a_j$ given by \eqref{aj}, 
\begin{equation}\label{E1-approx}
\E_1(x,y)=-x\sum_{j=1}^m\frac{a_j}{\log^j x}+O^\ast\left(\frac{m!}{\delta^{m+1}}\,\frac{x}{\log^{m+1} x}\right).
\end{equation}
To deal with $\E_2(x,y)$ we note that by induction on $n\geqslant 0$, we obtain the following anti-derivative formula with the coefficients $P(n,j)={n\choose j}j!$,
\begin{equation}\label{int-P(n,j)}
\int\frac{\log^n t}{t^2}\,\d t=-\frac{1}{t}\sum_{j=0}^n P(n,j)\log^{n-j}t.
\end{equation}
Since $y\leqslant x^\Delta$ we get $\frac{x}{y}\geqslant x^{1-\Delta}$. Thus, for any integer $n\geqslant 0$ we have
\[
\int_\frac{x}{y}^\infty\frac{\{t\}}{t^2}\,\log^n t\,\d t
\leqslant\int_{x^{1-\Delta}}^\infty\frac{\{t\}}{t^2}\,\log^n t\,\d t
<\int_{x^{1-\Delta}}^\infty\frac{\log^n t}{t^2}\,\d t
\]
By using \eqref{int-P(n,j)}, and assuming that $x\geqslant\e$, we get
\begin{align*}
\int_{x^{1-\Delta}}^\infty\frac{\log^n t}{t^2}\,\d t
&=\frac{\log^n x}{x^{1-\Delta}}\sum_{j=0}^n P(n,j)(1-\Delta)^{n-j}\frac{1}{\log^j x}\\
&<\frac{\log^n x}{x^{1-\Delta}}\sum_{j=0}^n P(n,j)
=\frac{\log^n x}{x^{1-\Delta}}\sum_{j=0}^n\frac{n!}{j!}<\e n!\,\frac{\log^n x}{x^{1-\Delta}}.
\end{align*}
Thus, for any integer $n\geqslant 0$ we obtain
\begin{equation}\label{int-x-y-bound}
\I_n(x,y):=\int_\frac{x}{y}^\infty\frac{\{t\}}{t^2}\,\log^n t\,\d t<\e n!\,\frac{\log^n x}{x^{1-\Delta}}.
\end{equation}
Applying \eqref{gxt-truncated-geometric-series} we get
\begin{align*}
\frac{\log x}{x}\,\E_2(x,y)
&=\int_\frac{x}{y}^\infty\frac{\{t\}}{t^2}\,g_x(t)\,\d t\\
&=\sum_{r=0}^{m-1}\frac{1}{\log^r x}\,\I_r(x,y)+O^\ast\left(\frac{1}{\delta^{m+1}\log^m x}\,\I_m(x,y)\right).
\end{align*}
Hence, by using \eqref{int-x-y-bound} we deduce that
\[
\E_2(x,y)<\left(\frac{\e m!}{\delta^{m+1}}+\e\sum_{r=0}^{m-1}r!\right)\frac{x^\Delta}{\log x}.
\]
Since $\sum_{r=0}^{m-1}r!\leqslant m!$, we obtain
\begin{equation}\label{E2-approx}
\E_2(x,y)=O^\ast\left(\left(1+\frac{1}{\delta^{m+1}}\right)\e m!\,\frac{x^\Delta}{\log x}\right).
\end{equation}
Combining \eqref{sum-li-x-y-E} with approximations \eqref{E1-approx} and \eqref{E2-approx} we obtain 
\begin{multline*}\label{sum-li-x-n}
\sum_{n\leqslant\frac{x}{y}}\li\left(\frac{x}{n}\right)
=\left[\frac{x}{y}\right]\li(y)+x(\log\log x-\log\log y)+x\sum_{j=1}^m\frac{a_j}{\log^j x}\\
+O^\ast\left(\frac{m!}{\delta^{m+1}}\,\frac{x}{\log^{m+1} x}+\left(1+\frac{1}{\delta^{m+1}}\right)\e m!\,\frac{x^\Delta}{\log x}\right).
\end{multline*} 
Now, to conclude the proof of \eqref{sum-pi-x-n} we need just to approximate the sum $\sum_{n\leqslant\frac{x}{y}}R\left(\frac{x}{n}\right)$. Since $n\leqslant\frac{x}{y}$ we have $\frac{x}{n}\geqslant y$. Thus,
\[
\sum_{n\leqslant\frac{x}{y}}R\left(\frac{x}{n}\right)\leqslant x\,\e^{-\frac{1}{3}\sqrt{\log y}}\sum_{n\leqslant\frac{x}{y}}\frac{1}{n}\leqslant x\,\e^{-\frac{1}{3}\sqrt{\log y}}\left(1+\log\frac{x}{y}\right).
\]
This completes the proof.
\end{proof}

\begin{proof}[Proof of Theorem \ref{sum-omega-full-explicit-thm}]
Considering the hyperbolic identity \eqref{sum-omega-Dirichlet} and approximations \eqref{sum-p-y-x-p-order-infinity} and \eqref{sum-pi-x-n} we get
\begin{equation}\label{sum-omega-full-h3xy}
\sum_{n\leqslant x}\o(n)=x\log\log x+Mx+x\sum_{j=1}^m\frac{a_j}{\log^j x}+O^\ast\left(h_3(x,y)\right),
\end{equation}
where
\[
h_3(x,y)=h_1(x,y)+h_2(x,y)+\left[\frac{x}{y}\right]\left(\li(y)-\pi(y)\right).
\]
By using \eqref{pi-li-bound} we deduce that 
\begin{align*}
\left[\frac{x}{y}\right]\left(\li(y)-\pi(y)\right)
&=\left[\frac{x}{y}\right]O^\ast\left(R(y)\right)\\
&=O^\ast\left(x\,\frac{R(y)}{y}\right)
=O^\ast\left(x\,\e^{-\frac{1}{3}\sqrt{\log y}}\right).
\end{align*}
Thus, \eqref{sum-omega-full-h3xy} holds with $h_3(x,y)=h_1(x,y)+h_2(x,y)+x\,\e^{-\frac{1}{3}\sqrt{\log y}}$, or with
\begin{multline*}\label{h3xy}
h_3(x,y)=\frac{m!}{\delta^{m+1}}\,\frac{x}{\log^{m+1} x}+\left(1+\frac{1}{\delta^{m+1}}\right)\e m!\,\frac{x^\Delta}{\log x}\\+x\,\e^{-\frac{1}{3}\sqrt{\log y}}\left(\log\frac{x}{y}+6\sqrt{\log y}+21\right)+y.
\end{multline*}
Now, we take $\delta=\Delta=\frac{1}{2}$, and hence $y=\sqrt{x}$. Note that the assumption $x\geqslant\e$ covers $x^\delta=\sqrt{x}>1.2$. Thus, we obtain \eqref{sum-omega-full-explicit}, and the proof is complete.
\end{proof}

\begin{proof}[Proof of Corollary \ref{sum-omega-explicit-order-2-cor}] We use \eqref{sum-omega-full-explicit} with $m=1$. Letting
\[
h(z)=z^4\e^{-\frac{\sqrt{2}}{6}z}\left(\frac{z^2}{2}+3\sqrt{2}z+21\right)+z^2\e^{-\frac{z^2}{2}}\left(z^2+5\e\right),
\]
we have
\[
h(\sqrt{\log x})=\frac{\log^2 x}{x}\left(\E_\o(x,1)-\frac{4x}{\log^2x}\right).
\]
By computation, we observe that $h(z)$ is decreasing for $z>23.97$, and $h(119.02511)<1<h(119.02510)$. When $x\geqslant\e^{14167}$ we have $\sqrt{\log x}\geqslant 119.02511$ and consequently, $h(\sqrt{\log x})<1$. Also, we note that $\left(1-\g\right)\frac{x}{\log x}>\frac{5x}{\log^2x}$ provided $x>\e^{5/(1-\g)}$, and this holds for the values of $x$ we work here. Hence, we conclude the proof.
\end{proof}

By using the following key result, Theorem \ref{sum-omega-full-explicit-thm} and Corollary \ref{sum-omega-explicit-order-2-cor} imply Theorem \ref{sum-Omega-full-explicit-thm} and Corollary \ref{sum-Omega-explicit-order-2-cor}, respectively.  

\begin{lemma}\label{J-explicit-bound-x-lemma}
For any $x\geqslant 2$ we have
\begin{equation}\label{J-explicit-bound-x}
\J(x):=\sum_{n\leqslant x}\left(\O(n)-\o(n)\right)=M''x+O^\ast\left(\frac{33\sqrt{x}}{\log x}\right).
\end{equation}
\end{lemma}
\begin{proof}
Let $\kappa(x)=\frac{25\sqrt{[x]}}{\log [x]}$. By using the double sided inequality \eqref{J-explicit-bound-n}, we deduce that
\begin{align*}
\J(x)
&=\sum_{k=1}^{[x]}\left(\O(k)-\o(k)\right)\\
&=M''[x]+O^\ast\left(\kappa(x)\right)
=M''x+O^\ast\left(\kappa(x)+M''\right).
\end{align*}
By computation, we observe that $\kappa(x)+M''<\frac{33\sqrt{x}}{\log x}$ for $x\geqslant 2$. 
\end{proof}

\begin{proof}[Proof of Corollary \ref{sum-Omega-explicit-order-2-cor}] Approximations \eqref{sum-omega-explicit-order-2} and \eqref{J-explicit-bound-x} imply
\[
\sum_{n\leqslant x}\O(n)=x\log\log x+M'x-\left(1-\g\right)\frac{x}{\log x}+O^\ast\left(\frac{5x}{\log^2x}+\frac{33\sqrt{x}}{\log x}\right).
\]
We note that 
\begin{equation}\label{33x-inq}
\frac{33\sqrt{x}}{\log x}<\frac{x}{\log^2 x},\qquad(x\geqslant 155652).
\end{equation}
This completes the proof.
\end{proof}

\section{Proof of conditional approximations}

To prove conditional results, under assuming the Riemann hypothesis, we reconstruct Lemma \ref{sum-p-y-x-p-order-infinity-lemma} and Lemma \ref{sum-pi-x-n-lemma}, replacing $R(x)$ by $\R(x)$. 

\begin{lemma}\label{sum-p-y-x-p-order-infinity-RH-lemma}
Assume that the Riemann hypothesis is true. Then, for any $x$ and $y$ satisfying $2\leqslant y\leqslant x$, we have
\begin{equation}\label{sum-p-y-x-p-order-infinity-RH}
\sum_{p\leqslant y}\left[\frac{x}{p}\right]=x\log\log y+Mx+O^\ast\left(\frac{x}{\sqrt{y}}\,(3\log y+4)+y\right).
\end{equation}
\end{lemma}
\begin{proof}
Note that
\[
\int_y^\infty\frac{\R(t)}{t^2}\,\d t=\frac{2\log y+4}{\sqrt{y}}.
\]
Thus, following similar argument as the proof of Lemma \ref{sum-p-y-x-p-order-infinity-lemma} and by using \eqref{pi-li-bound-RH}, we deduce that assuming RH, for any $y\geqslant 2$ we have
\[
\sum_{p\leqslant y}\frac{1}{p}=\log\log y+M+O^\ast\left(\frac{3\log y+4}{\sqrt{y}}\right).
\]
This completes the proof.
\end{proof}

\begin{lemma}\label{sum-pi-x-n-RH-lemma}
Assume that the Riemann hypothesis is true. Let $x$ and $y$ satisfy $x\geqslant\e$ and $1.2<x^\delta\leqslant y\leqslant x^\Delta<x$ for some fixed $\delta, \Delta\in(0,1)$. Then, we have
\begin{multline}\label{sum-pi-x-n-RH}
\sum_{n\leqslant\frac{x}{y}}\pi\left(\frac{x}{n}\right)
=\left[\frac{x}{y}\right]\li(y)+x(\log\log x-\log\log y)\\+x\sum_{j=1}^m\frac{a_j}{\log^j x}+O^\ast\left(\h_2(x,y)\right),
\end{multline}
where
\begin{multline*}\label{h2hat-xy}
\h_2(x,y)=\frac{m!}{\delta^{m+1}}\,\frac{x}{\log^{m+1} x}+\left(1+\frac{1}{\delta^{m+1}}\right)\e m!\,\frac{x^\Delta}{\log x}\\+\frac{2x}{\sqrt{y}}\left(\log y+2\right)+15\sqrt{x}\log x.
\end{multline*}
\end{lemma}
\begin{proof} Following similar argument as the proof of Lemma \ref{sum-pi-x-n-lemma}, we should approximate the sum $\sum_{n\leqslant\frac{x}{y}}\R\left(\frac{x}{n}\right)$, for which, we have
\begin{equation}\label{sum-R-hat}
\sum_{n\leqslant\frac{x}{y}}\R\left(\frac{x}{n}\right)=\sqrt{x}\log x\,\sum_{n\leqslant\frac{x}{y}}\frac{1}{\sqrt{n}}-\sqrt{x}\sum_{n\leqslant\frac{x}{y}}\frac{\log n}{\sqrt{n}}.
\end{equation}
Letting $f_0(t)=\frac{1}{\sqrt{t}}$ and $f_1(t)=\frac{\log t}{t}$, we observe that $f_0(t)$ is decreasing for $t\geqslant 1$, and with $t_0=\e^2\approxeq 7.39$, the function $f_1(t)$ is increasing for $1\leqslant t\leqslant t_0$ and decreasing for $t\geqslant t_0$. Moreover, 
\[
\max_{t\geqslant 1}f_1(t)=f_1(\e^2)=\frac{2}{\e}<1.
\] 
Thus, comparison of a sum and an integral of a monotonic function \cite[Theorem 0.4]{tenenbaum} implies that there exists $\t_0\in[0,1]$ such that
\[
\sum_{n\leqslant\frac{x}{y}}\frac{1}{\sqrt{n}}=1+\int_1^{[\frac{x}{y}]}f_0(t)\,\d t+\t_0\left(f_0\left(\left[\frac{x}{y}\right]\right)-1\right).
\]
Since $\max_{t\geqslant 1}f_0(t)=f_0(1)=1$, we get
\begin{equation}\label{sum-f0-int}
\sum_{n\leqslant\frac{x}{y}}\frac{1}{\sqrt{n}}
=\int_1^{[\frac{x}{y}]}f_0(t)\,\d t+O^\ast(3)
=\int_1^{\frac{x}{y}}f_0(t)\,\d t+O^\ast(4).
\end{equation}
Also, we write
\[
\sum_{n\leqslant\frac{x}{y}}\frac{\log n}{\sqrt{n}}
=\sum_{1<n\leqslant 7}\frac{\log n}{\sqrt{n}}+\frac{\log 8}{\sqrt{8}}+\sum_{8<n\leqslant\frac{x}{y}}\frac{\log n}{\sqrt{n}}.
\]
There exists $\t_1,\t_2\in[0,1]$ such that
\[
\sum_{1<n\leqslant 7}\frac{\log n}{\sqrt{n}}=\int_1^7 f_1(t)\,\d t+\t_1 f_1(7)=\int_1^7 f_1(t)\,\d t+O^\ast\left(\frac{2}{\e}\right),
\]
and
\begin{align*}
\sum_{8<n\leqslant\frac{x}{y}}\frac{\log n}{\sqrt{n}}
&=\int_8^{[\frac{x}{y}]} f_1(t)\,\d t+\t_2\left(f_1\left(\left[\frac{x}{y}\right]\right)-f_1(8)\right)\\
&=\int_8^{[\frac{x}{y}]} f_1(t)\,\d t+O^\ast\left(\frac{4}{\e}\right).
\end{align*}
Thus,
\[
\sum_{n\leqslant\frac{x}{y}}\frac{\log n}{\sqrt{n}}
=\int_1^{[\frac{x}{y}]} f_1(t)\,\d t+O^\ast\left(\eta\right)
=\int_1^{\frac{x}{y}} f_1(t)\,\d t+O^\ast\left(\eta+\frac{2}{\e}\right),
\]
where $\eta=\frac{6}{\e}+f_1(8)+\int_7^8 f_1(t)\,\d t\approxeq 3.68$. Since $\eta+\frac{2}{\e}<5$, we get
\begin{equation}\label{sum-f1-int}
\sum_{n\leqslant\frac{x}{y}}\frac{\log n}{\sqrt{n}}
=\int_1^{\frac{x}{y}} f_1(t)\,\d t+O^\ast(5).
\end{equation}
By computation, we have
\begin{multline*}
\sqrt{x}\log x\int_1^{\frac{x}{y}}f_0(t)\,\d t-\sqrt{x}\int_1^{\frac{x}{y}} f_1(t)\,\d t\\=\frac{2x}{\sqrt{y}}\left(\log y+2\right)-2\sqrt{x}\left(\log x+2\right).
\end{multline*}
Thus, considering the identity \eqref{sum-R-hat} and the approximations \eqref{sum-f0-int} and \eqref{sum-f1-int}, we deduce that
\[
\sum_{n\leqslant\frac{x}{y}}\R\left(\frac{x}{n}\right)
=\frac{2x}{\sqrt{y}}\left(\log y+2\right)+O^\ast\left(15\sqrt{x}\log x\right).
\]
This completes the proof.
\end{proof}

\begin{proof}[Proof of Theorem \ref{sum-omega-Omega-full-explicit-RH-thm}]
Considering the hyperbolic identity \eqref{sum-omega-Dirichlet} and approximations \eqref{sum-p-y-x-p-order-infinity-RH} and \eqref{sum-pi-x-n-RH} we get
\begin{equation}\label{sum-omega-full-h3xy-hat}
\sum_{n\leqslant x}\o(n)=x\log\log x+Mx+x\sum_{j=1}^m\frac{a_j}{\log^j x}+O^\ast\left(\h_3(x,y)\right),
\end{equation}
where
\[
\h_3(x,y)=\h_1(x,y)+\h_2(x,y)+\left[\frac{x}{y}\right]\left(\li(y)-\pi(y)\right),
\]
with $\h_1(x,y)=\frac{x}{\sqrt{y}}\,(3\log y+4)+y$. 
By using \eqref{pi-li-bound-RH} we deduce that 
\begin{align*}
\left[\frac{x}{y}\right]\left(\li(y)-\pi(y)\right)
&=\left[\frac{x}{y}\right]O^\ast\left(\R(y)\right)\\
&=O^\ast\left(\frac{x}{y}\,\R(y)\right)
=O^\ast\left(\frac{x\log y}{\sqrt{y}}\right).
\end{align*}
Thus, \eqref{sum-omega-full-h3xy-hat} holds with $\h_3(x,y)=\h_1(x,y)+\h_2(x,y)+\frac{x\log y}{\sqrt{y}}$, or with
\begin{multline*}\label{h3xy}
\h_3(x,y)=\frac{m!}{\delta^{m+1}}\,\frac{x}{\log^{m+1} x}+\left(1+\frac{1}{\delta^{m+1}}\right)\e m!\,\frac{x^\Delta}{\log x}\\+\frac{6x\log y}{\sqrt{y}}+\frac{8x}{\sqrt{y}}+15\sqrt{x}\log x+y.
\end{multline*}
Now, we take $\delta=\Delta=\frac{2}{3}$, and hence $y=x^{\frac{2}{3}}$. Note that the assumption $x\geqslant\e$ covers $x^\delta>1.2$. Thus, we obtain \eqref{sum-omega-full-explicit-RH}, and consequently we get \eqref{sum-Omega-full-explicit-RH} by using \eqref{J-explicit-bound-x}. The proof is complete.
\end{proof}

\begin{proof}[Proof of Corollary \ref{sum-omega-Omega-explicit-order-2-RH-cor}] We use \eqref{sum-omega-full-explicit-RH} with $m=1$. By computation, we observe that $\Eh_\o(x,1)<\frac{11x}{\log^2x}$ for $x\geqslant x_0$. Thus, we get \eqref{sum-omega-explicit-order-2-RH}, and consequently \eqref{sum-Omega-explicit-order-2-RH}, by using the approximation \eqref{J-explicit-bound-x} and the inequality \eqref{33x-inq}. Also, we note that 
\[
\left(1-\g\right)\frac{x}{\log x}>\frac{12x}{\log^2x}>\frac{11x}{\log^2x},
\] 
provided $x>\e^{12/(1-\g)}$. Since $x_0>\e^{12/(1-\g)}$, we conclude the proof. 
\end{proof}

\subsection*{Acknowledgement} The author is greatly indebted to Prof. Horst Alzer for suggesting the problem of finding global numerical bounds for the number-theoretic omega functions and for many stimulating conversations.

\end{document}